\documentclass[12pt]{amsart}
\usepackage[utf8]{inputenc}
\usepackage[T1]{fontenc}
\usepackage[margin=.9in]{geometry}
\usepackage{
amsmath,
amssymb,
amsthm,
amsrefs,
mathrsfs,
hyperref,
empheq, 
xcolor
}

\newtheorem{theorem}{Theorem}
\newtheorem{proposition}[theorem]{Proposition}
\newtheorem{lemma}[theorem]{Lemma}
\newtheorem{conjecture}[theorem]{Conjecture}
\newtheorem{definition}{Definition}

\newcommand{\Ov}{\mathscr{O}}
\newcommand{\al}{\alpha}

\newcommand{\la}{\lambda}

\newcommand{\E}{\mathbb{E}}
\newcommand{\R}{\mathbb{R}}
\newcommand{\C}{\mathbb{C}}
\newcommand{\ovr}{\overline}
\newcommand{\dd}{\mathrm{d}}
\newcommand{\OO}{\mathrm{O}}

\newcommand{\CUE}{\mathrm{CUE}}
\newcommand{\TUE}{\mathrm{TUE}}

\newcommand{\disteq}{\stackrel{d}{=}}
\newcommand{\be}{\begin{equation}}
\newcommand{\ee}{\end{equation}}

\title[Explicit formulas concerning eigenvectors
of weakly non-unitary matrices]{Explicit formulas concerning eigenvectors \\
of weakly non-unitary matrices}
\author{Guillaume Dubach}
\address{DMA, École Normale Supérieure -- PSL, 45 rue d'Ulm, F-75230 Cedex 5 Paris, France}
\email{guillaume.dubach@ens.fr}
\date{}

\begin{document}

\begin{abstract}
We investigate eigenvector statistics of the Truncated Unitary ensemble $\TUE(N,M)$ in the weakly non-unitary case $M=1$, that is when only one row and column are removed. We provide an explicit description of generalized overlaps as deterministic functions of the eigenvalues, as well as a method to derive an exact formula for the expectation of diagonal overlaps (or squared eigenvalue condition numbers), conditionally on one eigenvalue. This complements recent results obtained in the opposite regime when $M \geq N$, suggesting possible extensions to $\TUE(N,M)$ for all values of $M$.
\end{abstract}


\maketitle


\begin{figure}[h!]
\includegraphics[width=\textwidth]{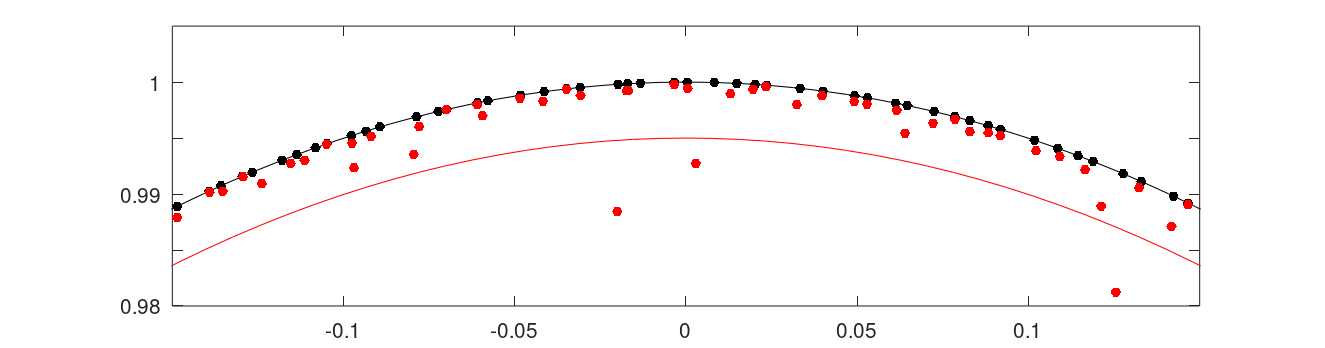}
\caption{Eigenvalues of a $\CUE(N+1)$ matrix (in black) and one of its $\TUE(N,1)$ truncations (in red), for $N=1000$. The smaller red circle is at distance $5/N$ from the unit circle.}
\label{fig:tue1}
\end{figure}


\section{Introduction}
\subsection*{General Presentation}
The ensemble $\TUE(N,M)$ of random matrices is defined as the distribution of an $N \times N$ submatrix of a unitary matrix $U$ of size $(N+M)\times(N+M)$, distributed according to the Haar measure on the unitary group. Eigenvalues of $\TUE(N,M)$ form a determinantal point process \cites{ZyczkowskiSommers, ForresterKrishnapur}, and are therefore well understood for all values of $M$; in particular, eigenvalues are almost surely distinct, and strictly inside the unit disk, their square radii being distributed as independent beta variables (Kostlan's theorem). Their associated eigenvectors are also, in principle, integrable; some aspects of their statistics have been studied recently \cite{DubachSpherical}, but for technical reasons these results only hold under the condition that $M \geq N$, i.e. when at least half of the lines and columns are truncated away from the original unitary matrix $U$. The present work is concerned with the opposite regime, namely the weakly non-unitary case $M=1$, when only one row and one column are removed from $U$. In this regime, the typical distance to the unit circle is of order $1/N$, as illustrated on Figure \ref{fig:tue1}. 
\medskip

Non-Hermitian matrices are endowed with left eigenvectors as well as right eigenvec\-tors, that can be chosen so as to form a bi-orthogonal family. Homogeneous quantities involving left and right eigenvectors are of particular interest; the particular quantities we study here are the generalized eigenvector overlaps, defined as follows:

\begin{definition}
If $L_i,R_i$ are left and right eigenvectors\footnote{We adopt the convention that $L_i=\langle L_i|$ is a row vector and $R_j = |R_j \rangle$ is a column vector, so that \eqref{def:biorthogonal} could be equivalently written $ L_i R_j = \delta_{ij} $ with the standard inner product, $\langle L_i | L_j \rangle $ is the same as $ L_i L_j^*$, etc.} of a matrix $G$ with distinct eigenvalues, normalized so as to form a bi-orthogonal family, that is, for all $i,j$,
\begin{equation}\label{def:biorthogonal}
\langle L_i \ | \ R_j \rangle = \delta_{ij}
\end{equation}
then for any integer $q \geq 1$ and $2q$-tuple of indices $i_1, j_1, \dots, i_q, j_q$, the corresponding generalized overlap, or $q$-overlap, is defined as the product
\begin{equation}
    \Ov_{i_1 j_1 \dots i_q j_q} 
    := \prod_{l=1}^q \langle L_{i_l} | L_{j_l} \rangle  \langle R_{j_l} | R_{i_{l+1}} \rangle,
\end{equation}
with the convention that $i_{q+1} = i_1$.
\end{definition}

These $q$-overlaps were introduced and studied in the complex Ginibre case by Crawford \& Rosenthal \cite{CrawfordRosenthal}. The case $q=1$, i.e. the matrix of overlaps $\Ov = (\Ov_{ij})$, is of particular relevance and has been the subject of much work in recent years, following the seminal papers of Chalker and Mehlig \cites{ChalkerMehlig, MehligChalker1, MehligChalker2}. Most importantly, the diagonal overlaps $\Ov_{ii} = \| L_i \|^2 \|R_i \|^2$, which are also the squared eigenvalue condition numbers, are of interest to both the mathematics and physics literature, see for instance \cites{BourgadeDubach, BelinschiNowakTarnowski, Fyodorov, FyodorovMehlig, FyodorovTarnowski, NowakTarnowski}.

The main goal of this paper is to establish that for $\TUE(N,M)$ with $M=1$, $q$-overlaps are given by deterministic functions of the eigenvalues, with no extra randomness involved. A similar situation was brought to light, for $q=1$, by Fyodorov \& Mehlig \cite{FyodorovMehlig} in the case of weakly non-Hermitian matrices, which exhibits some formal (but not complete) analogy with the present weakly non-unitary case. As a consequence, the study of $q$-overlaps in this context is reduced to questions concerning the eigenvalues of $\TUE(N,1)$, a well-understood determinantal point process, which we illustrate by providing an exact expression for the expectation of diagonal overlaps $\Ov_{ii}$.


\subsection*{Results}
Section \ref{gen_ov_sec} presents the derivation of the following explicit formula, for the generalized $q$-overlaps between the eigenvectors of a $\TUE(N,1)$ matrix, in terms of its eigenvalues.

\begin{theorem}\label{thm_overlaps_intro}
If $G$ is an $N \times N$ submatrix of a unitary matrix of size $(N+1)\times (N+1)$, with  distinct eigenvalues $\{ \la_1, \dots, \la_N \}$, then the $q$-overlaps between the eigenvectors of $G$ are given by
\begin{equation}\label{generalized_overlaps_formula}
\Ov_{i_1 j_1 \dots i_q j_q} = 
\prod_{l=1}^q \left( \frac{(1-|\la_{i_l}|^2)(1-|\la_{j_l}|^2)}{(1-\la_{i_l}\ovr{\la_{j_{l}}})(1-\la_{i_{l+1}}\ovr{\la_{j_l}})}
    \prod_{k \neq i_l} \frac{ \la_{i_l} \ovr{\la_k} - 1}
    { \la_{i_l} - \la_k }
    \prod_{k \neq j_l} \frac{ \ovr{\la_{j_l}} {\la_k} - 1}
    { \ovr{\la_{j_l} - \la_k} } \right).
\end{equation}
\end{theorem}

In the simplest case ($q=1$, $i=j$), this formula takes a form which can be explicitly integrated for a large class of determinantal point process with radially symmetric potential $V$. We present this general computation in Section \ref{computation_exp_sec}, and deduce from it the explicit formula for the expectation of the diagonal overlap $\Ov_{ii}$, conditionally on the location of the corresponding eigenvalue $\la_i$ in the unit disk\footnote{ Note that the eigenvalues $\lambda_1, \dots, \lambda_N$ are given in exchangeable random order, so that the conditioning only means, informally, that there is an eigenvalue at location $z_1$.}.

\begin{theorem}\label{thm:expO11}
If $G$ is distributed according to $\TUE(N,1)$, then the expectation of its diagonal overlap $\Ov_{11}$, conditionally on $\{ \la_1 = z_1 \}$ with $|z_1|<1$, is given by
\begin{equation}
\E_{\la_1=z_1} \Ov_{11} = \frac{(N+1)(1-|z_1|^2) + |z_1|^{2(N+1)} - 1}{ 1- |z_1|^{2(N+1)} - (N+1)(1-|z_1|^2) |z_1|^{2 N} }.
\end{equation}
In particular, if $|z_1| = r_1 < 1$ is fixed, then
\be
\frac{1}{N} \E_{|\la_1|=r_1} \Ov_{11} \xrightarrow[N \rightarrow \infty]{} 1- r_1^2 
\ee
and if $ \la_1 = z_1^{(N)}$ such that $1- |z_1^{(N)}|^2 \sim \kappa /N $,  then
\begin{equation}\label{univ_limit}
    \E_{ \la_1 = z_1^{(N)} } \Ov_{11}  \xrightarrow[N \rightarrow \infty]{} \frac{1- (1-\kappa) e^{\kappa}}{ e^{\kappa} -1 -\kappa }.
\end{equation}
\end{theorem}

Remarkably, \eqref{univ_limit} coincides with the corresponding limit for weakly non-Hermitian matrices in the appropriate regime, which follows from formula (4) and (5) in \cite{FyodorovMehlig}.


\subsection*{Conjecture and open questions}

Theorem \ref{thm_overlaps_intro} constitutes further supporting evidence for the following general conjecture:
\begin{conjecture}\label{conj_overlaps_all_ranks}
If $G$ is distributed according to $\TUE(N,M)$, the distribution of its diagonal overlaps is given by:
\begin{equation}
\Ov_{ii} \disteq \prod_{k \neq i} \left(1- \frac{ (\la_i - \overline{\la_i}) (\la_k - \overline{\la_k})}{|\la_i - \la_k|^2} Y^{(k)} \right)
\end{equation}
where $Y^{(k)}$ are i.i.d. real random variables following a $\beta_{1,M-1}$ distribution.
\end{conjecture}
Conjecture \ref{conj_overlaps_all_ranks} is known to be true for $M \geq N$ by \cite{DubachSpherical}, and for $M=1$ by Theorem~\ref{thm_overlaps_intro}, with the usual convention that $\beta_{1,0}$ is a Dirac with deterministic value $1$. 
Another question, in the cases where such decomposition is known to hold, would be to describe the limit of diagonal overlaps, and more generally of $q$-overlaps, when properly rescaled. It is to be expected that the answer to this question would exhibit a greater dependence on the parameter $M$, due to the important transition in the typical behavior of eigenvalues between the regimes where $M=\OO(1)$ and those with $M \geq N$. \medskip

\noindent \textit{Notation:} if $v$ is a vector, $v^{(k)}$ is the vector of its coordinates from $1$ to $k$, and $v^{[k]}$ the vector of its coordinates from $k$ to $n$.

\section{Generalized overlaps as functions of the eigenvalues}\label{gen_ov_sec}

The goal of this section is to establish Theorem \ref{thm_overlaps_intro}, a deterministic fact which follows from a series of linear algebra lemmata. The first lemma gives a formula for each column of an upper-triangular matrix $T$ such that $I-TT^*$ is a rank one positive operator.

\begin{lemma}
If $T$ is an $m \times m$ upper-triangular matrix with eigenvalues $ \la_1= T_{11}$, $\dots$, $\la_m = T_{mm} $, and  $v$ a column vector such that $TT^* = I_m - vv^*$, then for all $k=1, \dots, m$,
\begin{equation}
    |v_k|^2 = (1-|\la_k|^2) \prod_{l > k} |\la_l|^2
\end{equation}
and the $k$th column of $T$ is given by
\begin{equation}\label{columns_of_T}
\left(\begin{array}{c}
    \tau_k \\
    \la_k \\
    0
\end{array}\right)
\qquad \text{where} \qquad
    \tau_k 
    = - \frac{ \ovr{v_k} }{ \ovr{\la_k} \prod_{l>k} |\la_l|^2 } v^{(k-1)}
    = \frac{ |\la_k|^2-1 }{ \ovr{\la_k} } \frac{1}{v_k} v^{(k-1)}.
\end{equation}
\end{lemma}

\begin{proof}
We proceed by induction on $m$; the statement holds trivially when $m=1$. Let us assume it is true for matrices of size $m-1$, and consider the following matrix of size $m$:
$$
T_m=
\left(
\begin{array}{cc}
    T_{m-1} & \tau_m  \\
    0 & \la_m
\end{array}
\right).
$$
We have
$$
T_m T_m^* =
\left(
\begin{array}{cc}
    T_{m-1} T_{m-1}^* + \tau_m \tau_m^* & \ovr{\la}_m \tau_m  \\
    \la_m \tau_m^* & |\la_m|^2
\end{array}
\right),
$$
so that assuming that $T_m T_m^* = I - v v^*$ yields
\begin{empheq}[left=\empheqlbrace]{align}
\label{rec_1}  T_{m-1} T_{m-1}^*   &= I_{m-1} - v^{(m-1)} v^{(m-1)*} - \tau_m \tau_m^* \\
\label{rec_2}  \ovr{\la_m} \tau_m &= - \overline{v_m} v^{(m-1)} \\
\label{rec_3} |\la_m|^2 &= 1 - |v_m|^2.
\end{empheq}
Equation (\ref{rec_2}) together with \eqref{rec_3} yields the claim for $|v_m|^2$ and the last column $\tau_m$. Plugging these values in \eqref{rec_1} gives
\begin{equation}
T_{m-1} T_{m-1}^*   
= I_{m-1} - \left(1 + \frac{|v_m|^2}{|\la_m|^2}\right) v^{(m-1)} v^{(m-1)*}
= I_{m-1} - \frac{1}{|\la_m|^2} v^{(m-1)} v^{(m-1)*},
\end{equation}
which is the induction hypotheses with the vector $\frac{1}{\la_m} v^{(m-1)}$.
\end{proof}

The next lemma, established in a similar way, describes the rows of $T$ under a symmetric assumption.

\begin{lemma}
If $T$ is an $m \times m$ upper-triangular matrix with eigenvalues $ \la_1= T_{11}$, $\dots$, $\la_m = T_{mm} $, and  $w$ a row vector such that $T^* T = I_m - w^* w$, then for all $k=1, \dots, m$,
\begin{equation}
    |w_k|^2 = (1-|\la_k|^2) \prod_{l < k} |\la_l|^2
\end{equation}
and the $k$th row of $T$ is given by $(0, \la_k, \nu_k)$, where
\begin{equation}\label{rows_of_T}
\nu_k = - \frac{ \ovr{w_k} }{ \ovr{\la_k} \prod_{l < k} |\la_l|^2 } w^{[k+1]}
= \frac{|\la_k|^2-1}{\ovr{\la_k}} \frac{1}{w_k} w^{[k+1]}.
\end{equation}
\end{lemma}

\begin{proof}
We proceed by induction on $m$; the statement holds trivially when $m=1$. Let us assume it is true for matrices of size $m-1$, and consider the following matrix of size $m$:
$$
T_m=
\left(
\begin{array}{cc}
    \la_1 & \nu_1  \\
    0 & T_{m-1}
\end{array}
\right).
$$
We have
$$
T_m^* T_m =
\left(
\begin{array}{cc}
    |\la_1|^2 & \ovr{\la}_1 \nu_1  \\
    \la_1 \nu_1^* & T_{m-1}^* T_{m-1} + \nu_1^* \nu_1
\end{array}
\right),
$$
so that assuming that $T_m^* T_m = I - w^* w$ yields
\begin{empheq}[left=\empheqlbrace]{align}
\label{rec_1_w}  T_{m-1}^* T_{m-1} &= I_{m-1} - w^{[2]*} w^{[2]} - \nu_1^* \nu_1 \\
\label{rec_2_w}  \ovr{\la_1} \nu_1 &= - \ovr{w_1} w^{[2]} \\
\label{rec_3_w} |\la_1|^2 &= 1 - |w_1|^2.
\end{empheq}
Equation (\ref{rec_2_w}) together with \eqref{rec_3_w} yields the claim for $|v_1|^2$ and the first column $\nu_1$. Plugging these values in \eqref{rec_1_w} gives
\begin{equation}
T_{m-1} T_{m-1}^*   
= I_{m-1} - \left(1 + \frac{|w_1|^2}{|\la_1|^2}\right) w^{[2]} w^{[2]*}
= I_{m-1} - \frac{1}{|\la_1|^2} w^{[2]*} w^{[2]},
\end{equation}
which is the induction hypotheses with the vector $\frac{1}{\la_1} w^{[2]}$.
\end{proof}

\begin{proposition}\label{prop:eigenvectors_entries}
Let $T$ be an $N\times N$ upper-triangular matrix with distinct eigenvalues $\la_1, \dots, \la_N$ and such that $TT^* = I_N - v v^*$. The left and right eigenvectors of $T$ are given by:
\begin{equation}\label{left_right_eigenvectors}
L_i = (0, \dots , 0, l_{ii}=1, \dots l_{iN})
\qquad
\&
\qquad 
R_j = (r_{j1}, \dots, r_{jj}=1, 0, \dots, 0)^T
\end{equation}
where, for $i<j$,
\begin{equation}\label{left_right_eigenvectors_entries}
l_{ij} =  
\frac{ v_i }{ v_j }
\frac{ |\la_j|^2 -1 }{ \ovr{\la_j} \la_i -1 }
\prod_{l=i+1}^j
\frac{ \ovr{\la_{l}} \la_i - 1 }
{ \ovr{\la_l} (\la_i - \la_l) }
\qquad
\&
\qquad 
r_{ji} 
=
\frac{ \ovr{v_j} }{ \ovr{v_i} }
\frac{|\la_i|^2-1 }{ \ovr{\la_i} \la_j-1 }
\prod_{l=i+1}^j \la_l
\prod_{l=i}^{j-1} \frac{ \ovr{\la_l}\la_j - 1 }{ \la_j - \la_{l} }
.
\end{equation}
\end{proposition}

\begin{proof}
Eigenvectors can be chosen up to a multiplicative constant, in such a way that they form a bi-orthogonal family. For a triangular matrix $T$, this is achieved by the diagonal conditions $l_{ii}=1, r_{jj}=1$. The other coefficients are obtained by a straightforward recursion: for $i<j$,
$$
(L_i)_j = \frac{1}{\la_i - \la_j} L_i^{(j-1)} \tau_j
\qquad \& \qquad
(R_j)_i = \frac{1}{\la_j - \la_i} \nu_i R_{j}^{[i+1]}.
$$
We thus have, for $L_i$,
\begin{equation}\label{Li_first_step}
l_{i,i+1} 
= \frac{1}{\la_i - \la_{i+1}} L_i^{(i)} \tau_{i+1} 
= \frac{ T_{i,i+1} }{\la_i - \la_{i+1}}
= \frac{ |\la_{i+1}|^2-1 }{ \ovr{\la_{i+1}} (\la_i - \la_{i+1}) } \frac{ v_{i} }{ v_{i+1} },
\end{equation}
and then, with a two-steps recurrence using \eqref{columns_of_T},
\begin{align*}
l_{i,k+1} 
& = \frac{1}{\la_i - \la_{k+1}} L_i^{(k)} \tau_{k+1}  = \frac{ |\la_{k+1}|^2-1 }{\ovr{\la_{k+1}} (\la_i - \la_{k+1})} \frac{1}{ v_{k+1}}
\left(  L_i^{(k-1)} v^{(k-1)} + l_{ik} v_k \right) \\
& = \frac{ |\la_{k+1}|^2-1 }{\ovr{\la_{k+1}} (\la_i - \la_{k+1})} \frac{1}{ v_{k+1}}
\left(  \frac{ \ovr{\la_k } (\la_i - \la_{k}) v_k }{ |\la_k|^2-1 } + v_k \right) l_{ik} \\
& = \frac{ |\la_{k+1}|^2-1 }{ |\la_k|^2-1 }
\times
\frac{ v_k }{ v_{k+1}}
\times
\frac{ \ovr{ \la_k } \la_i -1 }{ 
\ovr{\la_{k+1}} (\la_i - \la_{k+1})} 
\times l_{ik}. 
\end{align*}
The result follows; likewise for $R_j$:
\begin{equation}\label{Rj_first_step}
r_{j,j-1} 
= \frac{1}{\la_j - \la_{j-1}} \nu_{j-1} R_j^{[j]} 
= \frac{ T_{j-1,j} }{\la_j - \la_{j-1}}
= \frac{ |\la_j|^2-1 }{ \ovr{\la_j} (\la_j - \la_{j-1}) } \frac{ v_{j-1} }{ v_j },
\end{equation} 
and then with another two-steps recursion, and by definition of $w$,
\begin{align*}
r_{j,k-1} 
& = \frac{1}{\la_j - \la_{k-1}} \nu_{k-1} R_j^{[k]}
= \frac{v_{k-1}}{\la_j - \la_{k-1}} w^{[k]} R_j^{[k]} \\
& = \frac{v_{k-1}}{\la_j - \la_{k-1}} \left( w^{[k+1]} R_j^{[k+1]} + w_k r_{jk} \right) \\
& = \frac{v_{k-1}}{\la_j - \la_{k-1}} \left( \frac{\la_j - \la_k}{v_k} r_{jk} + w_k r_{jk} \right) \\
& = \frac{ v_{k-1} }{ v_k } \times \frac{ \la_j \ovr{\la_k} -1 }{ \ovr{\la_k} (\la_j - \la_{k-1}) } \times r_{jk}
\end{align*}
and the result follows for $R_j$.
\end{proof}

\begin{definition}
We define
\begin{equation}
    \pi_{j+} = \prod_{l > j} \frac{ \la_j \ovr{\la_l} -1 }
    { \la_j - \la_l },
    \qquad
    \pi_{j-} = \prod_{l < j} \frac{ \la_j \ovr{\la_l} -1 }
    { \la_j - \la_l },
    \qquad
    \pi_{j} = \pi_{j+} \pi_{j-} 
    = \prod_{l \neq j} \frac{ \la_j \ovr{\la_l} - 1}
    { \la_j - \la_l }.
\end{equation}
\end{definition}

\begin{proposition}\label{scalar_products}
In the same context as above, the scalar products between left (resp. right) eigenvectors of $T$ are given by the formulas
\be\label{LiLj}
\langle L_i | L_j \rangle 
= 
\frac{v_i}{v_j}
\frac{|\la_j|^2 - 1}{ \la_i \ovr{\la_j} - 1 }
\left( \prod_{l=i \wedge j +1}^{i \vee j} \frac{1}{\ovr{\la_l}} \right)
\pi_{i+} \ovr{\pi_{j+}}, 
\ee

\be\label{RjRi}
\langle R_j | R_i \rangle 
= 
\frac{ v_j }{ v_i }
\frac{|\la_i|^2 - 1}{ \la_i \ovr{\la_j} - 1 }
\left( \prod_{l= i \wedge j+1}^{i \vee j} \ovr{\la_l} \right)
\pi_{i-} \ovr{\pi_{j-}}.
\ee
\end{proposition}

\begin{proof}
We perform both computations assuming $i<j$; the complementary statement follows by sesquilinearity. First, note that for any $k \geq i$,
$T$ being triangular, $L_i^{(k)}$ is a left-eigenvector of the smaller matrix $T^{(kk)}$ with eigenvalue $\la_i$, so that if $k \geq j >i$,
\be
L_i^{(k)} v^{(k)} v^{(k)*} L_j^{(k)*} = L_i^{(k)} \left(I - T^{(kk)} T^{(kk)*} \right) L_j^{(k)*}
= (1 - \la_i \ovr{\la}_j ) L_i^{(k)} L_j^{(k)*}.
\ee
Using \eqref{columns_of_T} and \eqref{Li_first_step} for $L_i$ and $L_j$ gives, still for $k\geq j>i$
\begin{align*}
l_{i,k+1} \ovr{l_{j,k+1}}
& =
\frac{ 1-|\la_{k+1}|^2 }{(\la_i - \la_{k+1})(\overline{\la_j - \la_{k+1}})   |\la_{k+1}|^2 |v_{k+1}|^2 }
L_i^{(k)} v^{(k)} v^{(k)*} L_j^{(k)*} \\
& =
\frac{ 1 }{(\la_i - \la_{k+1})(\overline{\la_j - \la_{k+1}})  \prod_{l=k+1}^N |\la_{l}|^2 }
L_i^{(k)} v^{(k)} v^{(k)*} L_j^{(k)*} \\
& = \frac{ 1 - \la_i \ovr{\la}_j }{(\la_i - \la_{k+1})(\overline{\la_j - \la_{k+1}})  \prod_{l=k+1}^N |\la_{l}|^2 }
L_i^{(k)} L_j^{(k)*}
\end{align*}

\begin{align*}
    \langle L_i^{(k+1)} | L_j^{(k+1)} \rangle & = L_i^{(k)} L_j^{(k)*} + l_{i,k+1} \ovr{l_{j,k+1}}
    =  L_i^{(k)} L_j^{(k)*} \left( 1 + \frac{ (1-|\la_{k+1}|^2)  (1 - \la_i \ovr{\la}_j ) }{ (\la_i - \la_{k+1}) ( \ovr{\la_j - \la_{k+1}}) } \right) \\
    & =
    L_i^{(k)} L_j^{(k)*}
    \frac{ (1-\la_i \ovr{\la_{k+1}})(1-\ovr{\la_j} \la_{k+1}) }{ (\la_i - \la_{k+1}) ( \ovr{\la_j - \la_{k+1}}) }
\end{align*}

It is straightforward to check that the initial case for $k=j$, with, according to Proposition \ref{prop:eigenvectors_entries}
\be
\langle L_i^{(j)} | L_j^{(j)} \rangle = l_{ij} = \frac{ v_i }{ v_j }
\frac{ |\la_j|^2 -1 }{ \ovr{\la_j} \la_i -1 }
\prod_{l=i+1}^j
\frac{ \ovr{\la_{l}} \la_i - 1 }
{ \ovr{\la_l} (\la_i - \la_l) }
\ee

For the right eigenvectors, we do a backward recursion, starting with
\be
\langle R_j^{[i]} \ | \ R_i^{[i]} \rangle
= \ovr{r_{ji}} 
= \frac{v_j}{v_i} \frac{|\la_i|^2-1}{\ovr{\la_j} \la_i -1} \left( \prod_{l=i+1}^j \ovr{\la_l} \right)
\prod_{l=i}^{j-1} \frac{\la_l \ovr{\la_j} - 1}{\ovr{\la_j - \la_l}}
\ee
and then similarly to complete the $\pi_{i-}, \pi_{j-}$ terms.
\end{proof}

Theorem \ref{thm_overlaps_intro} follows, by concatenating such terms as \eqref{LiLj} and \eqref{RjRi}.

\begin{proof}[Proof of Theorem \ref{thm_overlaps_intro}]
If $T$ is a Schur form of the matrix $G$, then there are a column vector $v$, a row vector $w$ and a number $\al$ such that
\begin{equation}
\left(
    \begin{array}{cc}
    T    & v  \\
    w    & \al
    \end{array}
\right) \in U_{N+1},
\end{equation}
which in particular yields the relations
\begin{equation}
    TT^* = I_N - vv^* \qquad \& \qquad T^* T = I_N - w^* w,
\end{equation}
and moreover, we assume that the eigenvalues of $T$ are distinct; so that $T$ meets the requirements of the above theorems. Plugging equations \eqref{LiLj}, \eqref{RjRi} in the definition
\begin{equation}
    \Ov_{i_1 j_1 \dots i_q j_q} 
    := \langle L_{i_1} | L_{j_1} \rangle  \langle R_{j_1} | R_{i_2} \rangle \cdots  \langle L_{i_q} | L_{j_q} \rangle  \langle R_{j_q} | L_{i_1} \rangle
\end{equation}
yields the claim, as the terms $v_i$, $v_j$ as well as other terms cancel exactly, and each $\pi_{i+}$ (resp. $\ovr{\pi_{j_+}}$) has a corresponding $\pi_{i-}$ (resp. $\ovr{\pi_{j_-}}$). We find:
\begin{equation}\label{generalized_overlaps_formula_2}
\Ov_{i_1 j_1 \dots i_q j_q} = 
\prod_{l=1}^q \frac{(1-|\la_{i_l}|^2)(1-|\la_{j_l}|^2)}{(1-\la_{i_l}\ovr{\la_{j_{l}}})(1-\la_{i_{l+1}}\ovr{\la_{j_{l}}})}
 \pi_{i_l} \ovr{\pi_{j_l}},
\end{equation}
which is the claim.
\end{proof}

The expressions for $q=1$ follow directly, leading to expressions analogous to the result from \cite{FyodorovMehlig} for weakly non-Hermitian matrices. Indeed, diagonal overlaps are given by the product:
\begin{equation}\label{diag_overlap}
\Ov_{ii} = |\pi_i|^2
= \prod_{k \neq i} \left| \frac{ \la_i \ovr{\la_k} -1}{\la_i - \la_k} \right|^2
= \prod_{k \neq i} \left(1 + \frac{ ( 1 - |\la_k|^2) ( 1 - |\la_i|^2)}{|\la_i - \la_k|^2} \right),
\end{equation}
whereas off-diagonal overlaps are given by:
\begin{equation}
\Ov_{ij} = 
- \frac{ ( 1 - |\la_i|^2) ( 1 - |\la_j|^2) }{ |\la_i - \la_j|^2 }
 \prod_{k \neq i,j} \left(1 + \frac{ ( 1 - |\la_k|^2) ( 1 - \la_i \overline{\la_j})}{(\la_i - \la_k)(\overline{\la_j - \la_k})} \right) \qquad \forall i \neq j.
\end{equation}

\section{The expectation of diagonal overlaps}\label{computation_exp_sec}

\subsection{A computation for general potential}
In this section, we work with a general function $V=V_N: \R_+ \rightarrow \R \cup \{+\infty \}$, and define the generalized Gamma and Meijer functions as
\begin{equation}
\Gamma_V(\alpha) := \int_{\R_+} t^{\al - 1} e^{-V(t)} \dd t,
\qquad
G_{V} (k) = \Gamma_V(1) \cdots \Gamma_V(k).
\end{equation}
We also define the partial sums
\begin{equation}
e_V^{(m)}(X) := \sum_{k=0}^m \frac{X^k}{\Gamma_V (k+1)}.
\end{equation}
The method that follows applies to any determinantal point process with radial symmetry in $\C$, with joint density given by
\begin{equation}\label{Beta_ens_2}
\frac{1}{Z_N} \prod_{1 \leq i<j \leq N} |\la_i - \la_j|^2 e^{- \sum_{i=1}^N V(|\la_i|^2)} ,
\end{equation}
with respect to the Lebesgue measure $\dd m$ on $\mathbb{C}$, such that the integral quantities of interest are finite; in particular, we then have 
\be
Z_N = G_V(N).
\ee
As noted in \cite{DubachPowers}, Kostlan's theorem and related identities hold, as well as the following formula for product statistics conditionally on one eigenvalue:

\begin{proposition}\label{generalV_prostat} Conditionally on $\{ \la_1 = z_1 \}$, for any function $g$, the following identity holds:
$$
\E_{\la_1=z_1} \left( \prod_{i=2}^N g(\la_i) \right)
=
\frac{1}{Z_N^{(1)}}
\det \left( \int_{\mathbb{C}} z^{i-1} \overline{z}^{j-1} |z_1 - z|^2 g(z) e^{-V(|z|^2)} \dd m(z) \right)_{i,j=1}^{N-1}
$$
where
$$
Z_N^{(1)} = G_V(N) e_V^{(N-1)} (|z_1|^2).
$$
\end{proposition}
This key fact could be justified in different ways; we give the following proof as an illustration of the tridiagonal determinant structure (also known as \textit{continuant}), that is also the key to the main result. 

\begin{proof} The determinant comes from Andr\'eief identity applied to the conditioned measure; it remains to compute the normalization constant by setting $g=1$. We find
$$
Z_N^{(1)}
= \det \left( f_{ij} \right)_{i,j=1}^{N-1} =: D_{N-1}
$$
where the determinants $D_N$ are Hermitian, tridiagonal and nested, with
\begin{align*}
f_{kk} & = |z_1|^2 \Gamma_V(k) + \Gamma_V(k+1) \\
f_{k+1,k} & = \overline{f_{k,k+1}} = - z_1 \Gamma_V(k+1).
\end{align*}
This gives the initial values $D_0=1$, $$ D_{1} = |z_1|^2 \Gamma_V(1) + \Gamma_V(2) = G_V(2) e_V^{(1)} (|z_1|^2), $$
and the induction
\begin{align*}
D_{k}
& = f_{kk} D_{k-1} - |f_{k,k-1}|^2 D_{k-2} \\
& = \left( |z_1|^2 \Gamma_V(k) + \Gamma_V(k+1) \right) D_{k-1} - |z_1|^2 \Gamma_V(k)^2 D_{k-2}
\end{align*}
whose solution is the formula provided for $Z_N^{(1)}$. One way to see this is to introduce the sequence $\tilde{D}$ such that $D_{k} = G_V(k+1) \tilde{D}_k$, which yields
$$
\tilde{D}_{k}
= \left( |z_1|^2 \frac{\Gamma_V(k)}{\Gamma_V(k+1)} + 1 \right) \tilde{D}_{k-1} - |z_1|^2 \frac{\Gamma_V(k)}{\Gamma_V(k+1)} \tilde{D}_{k-2}
$$
so that
$$
\tilde{D}_{k} - \tilde{D}_{k-1} 
= |z_1|^2 \frac{\Gamma_V(k)}{\Gamma_V(k+1)}  \left( \tilde{D}_{k-1} -  \tilde{D}_{k-2} \right)
= \dots =  \frac{|z_1|^{2k}}{\Gamma_V(k+1)},
$$
whence $\tilde{D}_k = e_V^{k} (|z_1|^2)$.
\end{proof}
We now apply Proposition \ref{generalV_prostat} to functions $g$ of a particular form.

\begin{proposition}\label{magic_trick} In the ensemble defined by (\ref{Beta_ens_2}), conditionally on $\{ \la_1 = z_1 \}$, for any parameter $\al$,
\be
\E_{\la_1=z_1} \left( 
\prod_{i=2}^N 
\left( 1 + \frac{(\al-1) |z_1|^2 + (\al^{-1} -1) |\la_k|^2}{|z_1 - \la_k |^2} \right) \right)
= \frac{e_V^{N-1} (\al^2 |z_1|^2)}{\al^{N-1} e_V^{N-1} (|z_1|^2)}.
\ee
In particular, for $\al= |z_1|^{-2}$ we find that
\be
\E_{\la_1=z_1} \left( 
\prod_{k=2}^N 
\left( 1 + \frac{(1 -  |z_1|^2)(1 -  |\la_k|^2)}{|z_1 - \la_k |^2} \right) \right)
= \frac{|z_1|^{2 (N-1)} e_V^{N-1} (|z_1|^{-2})}{ e_V^{N-1} (|z_1|^2)}.
\ee
\end{proposition}

\begin{proof}
Proposition \ref{generalV_prostat} implies that, for generic $A,B$,
\begin{align*}
\E_{\la_1=z_1} & \left( 
\prod_{k=2}^N 
\left( 1 + \frac{A + B |\la_k|^2}{|z_1 - \la_k |^2} \right) \right) \\
& =
\frac{1}{Z_N^{(1)}}
\det \left( \int_{\mathbb{C}} z^{i-1} \overline{z}^{j-1} (|z_1 - z|^2 + A + B |z|^2 ) e^{-V(|z|^2)} \dd m(z) \right)_{i,j=1}^{N-1} \\
& = \frac{1}{Z_N^{(1)}} D_{N-1}
\end{align*}
with initial terms 
\be
D_0=1, \quad
D_1 = (A+|z_1|^2) \Gamma_V(1) + (B+1) \Gamma_V(2),
\ee
and the recursion
\begin{align*}
D_{k} &
= f_{kk} D_{k-1} - |f_{k,k-1}|^2 D_{k-2} \\
& = \left( (A+|z_1|^2) \Gamma_V(k) + (B+1) \Gamma_V(k+1) \right) D_{k-1} - |z_1|^2 \Gamma_V(k)^2 D_{k-2}
\end{align*}
The formula obviously holds for $z_1=0$. If $z_1 \neq 0$, let us define $\al, \beta$ such that
$$
A+|z_1|^2 = \al |z_1|^2, \qquad
B+1=\beta,
$$
so that the recurrence is
\begin{align*}
D_{k} = \left( \al|z_1|^2 \Gamma_V(k) + \beta \Gamma_V(k+1) \right) D_{k-1} - |z_1|^2 \Gamma_V(k)^2 D_{k-2}.
\end{align*}
Now, if $ D_k = G_V(k+1) \tilde{D}_k$, we find
\begin{align*}
\tilde{D}_{k} = \left( \al |z_1|^2 \frac{\Gamma_V(k)}{\Gamma_V(k+1)} + \beta \right) \tilde{D}_{k-1} - |z_1|^2 \frac{\Gamma_V(k)}{\Gamma_V(k+1)} \tilde{D}_{k-2}
\end{align*}
so that
$$
\tilde{D}_k - \beta \tilde{D}_{k-1} = \al |z_1|^2 \frac{\Gamma_V(k)}{\Gamma_V(k+1)} \left( \tilde{D}_{k-1} - \frac{1}{\al} \tilde{D}_{k-2} \right).
$$
In order for this to have a simple solution, we assume $\al \beta =1$, which implies the condition $B= \al^{-1} -1$ in the statement. Under this assumption, we solve by a direct recurrence:
\begin{align}
\tilde{D}_k - \beta \tilde{D}_{k-1} &
= \al^{k-1} |z_1|^{2(k-1)} \frac{\Gamma_V(2)}{\Gamma_V(k+1)}
\left( \tilde{D}_{1} - \frac{1}{\al} \tilde{D}_{0} \right) \\
& = \al^{k-1} |z_1|^{2(k-1)} \frac{\Gamma_V(2)}{\Gamma_V(k+1)}
\left( \frac{\al |z_1|^2}{\Gamma_V (2)} + \frac{1}{\al \Gamma_V(1)} - \frac{1}{\al \Gamma_{V} (1)} \right) \\
& = \frac{ \al^{k} |z_1|^{2k}}{\Gamma_V(k+1)}
\end{align}
From which, using again $\al \beta=1$, one deduces that $\tilde{D}_k = \al^{-k} e_V^{(k)} (\al^2 | z_1|^2)$. The result follows.
\end{proof}

\subsection{Application to diagonal overlaps}

We now prove Theorem \ref{thm:expO11}, using the explicit form of the joint density from \cite{ZyczkowskiSommers} and the formula \eqref{diag_overlap} for diagonal overlaps that follows from Theorem \ref{thm_overlaps_intro}.


\begin{proof}[Proof of Theorem \ref{thm:expO11}]
We apply Proposition \ref{magic_trick} to $\TUE(N,1)$. In this case, $V(t) = 1$ on $[0,1)$ and $+\infty$ on $[1, +\infty)$. This gives the generalized Gamma function
\begin{equation}
\Gamma_V (\al) = \frac{1}{\al}
\end{equation}
and, for every $m$, we compute
\begin{align*}
e_V^{(m)} (X) = \sum_{k = 0}^m (k+1) X^{k} & = \partial_X \left( \frac{1-X^{m+2}}{1-X} \right) \\
& = \frac{ - (m+2) X^{m+1} (1-X) + 1 - X^{m+2} }{(1-X)^2},
\end{align*}
as well as
\be
X^{m} e_V^{(m)} \left( 1/X \right) = \frac{(m+2) (1-X) +  X^{m+2} -1 }{(1-X)^2}.
\ee
The expectation of the diagonal overlap is given by the following expression, obtained from the above with $m=N-1$, and evaluated at $X=|z_1|^2$:
\be
\frac{
X^{N-1} e_V^{(N-1)} \left( 1/X \right)
}{
e_V^{(N-1)} (X)
}
= \frac{(N+1) (1-X) +  X^{N+1} -1 }{ 1 - X^{N+1} - (N+1) X^{N} (1-X) },
\ee
the consequences follow by a direct computation in the relevant regimes.
\end{proof}


\section*{Acknowledgments}

\noindent The author acknowledges funding from the European Union's Horizon 2020 research and innovation programme under the Marie Sk{\l}odowska-Curie Grant Agreement No. 754411. His gratitude also goes to Jana Reker as well as to the anonymous referee for their careful proofreading and comments.

\end{document}